\documentclass[a4paper,american]{amsart}
\usepackage{mathptmx}
\usepackage[T1]{fontenc}
\usepackage[latin9]{inputenc}
\usepackage{babel}
\usepackage{refstyle}
\usepackage{units}
\usepackage{amsmath}
\usepackage{amsthm}
\usepackage{amssymb}
\usepackage[unicode=true,pdfusetitle,
 bookmarks=true,bookmarksnumbered=false,bookmarksopen=false,
 breaklinks=false,pdfborder={0 0 1},backref=false,colorlinks=false]
 {hyperref}
 \hypersetup{
  bookmarksopen}

\makeatletter

\AtBeginDocument{\providecommand\secref[1]{\ref{sec:#1}}}
\AtBeginDocument{\providecommand\thmref[1]{\ref{thm:#1}}}
\AtBeginDocument{\providecommand\lemref[1]{\ref{lem:#1}}}
\AtBeginDocument{\providecommand\propref[1]{\ref{prop:#1}}}
\pdfpageheight\paperheight
\pdfpagewidth\paperwidth

\RS@ifundefined{subref}
  {\def\RSsubtxt{section~}\newref{sub}{name = \RSsubtxt}}
  {}
\RS@ifundefined{thmref}
  {\def\RSthmtxt{theorem~}\newref{thm}{name = \RSthmtxt}}
  {}
\RS@ifundefined{lemref}
  {\def\RSlemtxt{lemma~}\newref{lem}{name = \RSlemtxt}}
  {}

\theoremstyle{plain}
\numberwithin{equation}{section}
\numberwithin{figure}{section}
\numberwithin{table}{section}
\usepackage{refstyle}
\newtheorem{stthm}{\protect\theoremname}[section]

\newtheorem{spprop}{\protect\propositionname}[section]

 \newref{sec}{%
 name = \RSsectxt,
 names = \RSsecstxt,
 lsttxt = \RSlsttxt,
 lsttwotxt = \RSlsttwotxt,
 refcmd = \ref{#1}}
 \theoremstyle{definition}
 \newtheorem{sddefn}{\protect\definitionname}[section]
  \newref{def}{%
  name = \definitionname~,
  lsttxt = \RSlsttxt,
  lsttwotxt = \RSlsttwotxt,
  refcmd = {\ref{#1}}}
 \theoremstyle{remark}
 \newtheorem*{rem*}{\protect\remarkname}
 \newref{rem}{%
 name = \remarkname~,
 lsttxt = \RSlsttxt,
 lsttwotxt = \RSlsttwotxt,
 refcmd = {\ref{#1}}}
 \newref{thm}{%
 name = \theoremname~,
 lsttxt = \RSlsttxt,
 lsttwotxt = \RSlsttwotxt,
 refcmd = {\ref{#1}}}
 \theoremstyle{plain}
 \newtheorem{sllem}{\protect\lemmaname}[section]
 \newref{lem}{%
 name = \lemmaname~,
 lsttxt = \RSlsttxt,
 lsttwotxt = \RSlsttwotxt,
 refcmd = {\ref{#1}}}
 \usepackage{enumitem}
 \usepackage{enumitem}
 \theoremstyle{remark}
 \newtheorem{srrem}{\protect\remarkname}[section]
 \newref{rem}{%
 name = \remarkname~,
 lsttxt = \RSlsttxt,
 lsttwotxt = \RSlsttwotxt,
 refcmd = {\ref{#1}}}
 \newref{prop}{%
 name = \propositionname~,
 lsttxt = \RSlsttxt,
 lsttwotxt = \RSlsttwotxt,
 refcmd = {\ref{#1}}}
 \theoremstyle{plain}
 \newcounter{LonSP}[spprop]
 \newtheorem{llemsp}[LonSP]{\protect\lemmaname}
 \newref{lem}{%
 name = \lemmaname~,
 lsttxt = \RSlsttxt,
 lsttwotxt = \RSlsttwotxt,
 refcmd = {\ref{#1}}}
 \theoremstyle{plain}
 \newtheorem{stlem}[stthm]{\protect\lemmaname}
 \newref{lem}{%
 name = \lemmaname~,
 lsttxt = \RSlsttxt,
 lsttwotxt = \RSlsttwotxt,
 refcmd = {\ref{#1}}}

\newcommand\mynobreakpar{\par\nobreak\@afterheading}
\usepackage{bbm, yhmath}
\AtBeginDocument{}
\newtheorem*{mainthm}{\upshape{\textbf{Main Theorem}}}
\newref{mthm}{%
 name = Main Theorem,
 refcmd = {\ref{#1}}}
\usepackage{esvect}

\makeatother

 \providecommand{\definitionname}{Definition}
 \providecommand{\lemmaname}{Lemma}
 \providecommand{\propositionname}{Proposition}
 \providecommand{\remarkname}{Remark}
 \providecommand{\theoremname}{Theorem}

\begin{document}

\title{Zeros sets of $H^{p}$ Functions in Lineally Convex Domains of Finite
Type in $\mathbb{C}^{n}$}

\author{P. Charpentier \& Y. Dupain}
\begin{abstract}
In this note we extend N. Th. Varopoulos result on zero sets of $H^{p}$
functions of strictly pseudo-convex domains in $\mathbb{C}^{n}$ to
lineally convex domains of finite type.
\end{abstract}



\address{P. Charpentier, Universit\'e Bordeaux I, Institut de Math\'ematiques
de Bordeaux, 351, Cours de la Lib\'eration, 33405, Talence, France}

\email{P. Charpentier: philippe.charpentier@math.u-bordeaux.fr}

\maketitle

\section{Introduction}

The study of the zero-sets of holomorphic functions in a given class
of a smoothly bounded domain in $\mathbb{C}^{n}$ is a very classical
problem which has been intensively studied. When $n=1$ those sets
are characterized, by the Blaschke condition, for the Nevanlinna and
Hardy classes, but when $n\geq2$, the situation is more complicated.

Such characterizations are only known for the Nevanlinna under additional
hypothesis on the domain: G. M. Henkin (\cite{Hen75}) and H. Skoda
(\cite{Sko76} independently obtained the case of strictly pseudo-convex
domains, D. C. Chang A. Nagel and E. Stein (\cite{Chang-Nagel-Stein})
proved the same result for pseudo-convex domains of finite type in
$\mathbb{C}^{2}$, and much later the case of convex domains of finite
type (\cite{Bruna-Charp-Dupain-Annals,Cumenge-Navanlinna-convex,DM01}),
and, recently, the case of lineally convex domains of finite type
were obtained (\cite{CDMb}).

In \cite{VarHp80} (see also \cite{AnCa-Hp-90}) N. Th. Varopoulos
proved that, in a strictly pseudo-convex domain, a divisor satisfying
a special Carleson condition is always defined by a function in some
Hardy space $H^{p}(\Omega)$. Tentatives to generalize this result,
for example to convex domains of finite type, were done in \cite{BrGrHp99}
and \cite{Ngu01}, but some gaps in the proofs leave the problem open
until a recent paper of W. Alexandre (\cite{Ale17}). In this last
paper, the author make a strong use of the estimates of the Bergman
metric obtained by Mc. Neal.

In this note, we show that Varopoulos result extends to lineally convex
domains of finite type with a classical method using only the anisotropic
geometry, described in \cite{Conrad_lineally_convex}, of those domains.

\section{Main results}

Let us first recall the definition of a lineally convex domain:
\begin{sddefn}
A domain $\Omega$ in $\mathbb{C}^{n}$, with smooth boundary is said
to be lineally convex at a point $p\in\partial\Omega$ if there exists
a neighborhood $W_{p}$ of $p$ such that, for all point $z\in\partial\Omega\cap W_{p}$,
\[
\left(z+T_{z}^{10}\right)\cap(\Omega\cap W_{p})=\emptyset,
\]
where $T_{z}^{10}$ is the holomorphic tangent space to $\partial\Omega$
at the point $z$.
\end{sddefn}

In all the paper, we assume that $\partial\Omega$ is of finite type
and lineally convex at every point of $\partial\Omega$. We may assume
that there exists a $\mathcal{C}^{\infty}$defining function $\rho$
for $\Omega$ and a number $\eta_{0}>0$ such that $\nabla\rho(z)\neq0$
at every point of $W=\left\{ -\eta_{0}\leq\rho(z)\leq\eta_{0}\right\} $
and the level sets 
\[
\left\{ z\in W\mbox{ such that }\rho(z)=\eta\right\} ,
\]
are lineally convex of finite type.

We thus assume that the defining function $\rho$ satisfies this hypothesis.

In the next section, for $z\in\Omega$ in a fixed small neighborhood
$V$ of $\partial\Omega$, we recall the definition of the two fundamental
quantities $\tau(z,u,\delta)$, for $0<\delta\leq\delta_{0}$, $\delta_{0}>0$
depending only on $\Omega$, and $u$ a non zero complex vector and
$k(z,u)=\frac{\delta_{\Omega}(z)}{\tau\left(z,u,\delta_{\Omega}(z)\right)}$
where $\delta_{\Omega}(z)$ denotes the distance of $z$ to the boundary
of $\Omega$. The lineal convexity hypothesis implies that $(z,u)\mapsto k(z,u)$
is a continuous function in $V\times\mathbb{C}_{*}^{n}$ and, for
$0<\delta_{1}<\delta_{0}$ and $K<+\infty$, there exists constants
$c>0$ and $C<+\infty$ such that for $z\in V\cap\left\{ \delta_{\Omega}(\zeta)\geq\delta_{1}\right\} $
and $\frac{1}{K}\leq\left|u\right|\leq K$, $c\leq k(z,u)\leq C$.
So, if $u(z)$ is a continuous vector field in $\Omega$, $\frac{1}{K}\leq\left|u\right|\leq K$,
$k(z,u(z))$ can be extended to $\Omega$ in a continuous function
satisfying $c\leq k(z,u)\leq C$ in $\Omega\cap\left\{ \delta_{0}\geq\delta_{\Omega}(\zeta)\geq\delta_{1}\right\} $.

To state our main result we have to recall the notion of Carleson
measure in our context (see \cite{ChDu18}): a bounded measure $\mu$
in $\Omega$ is called a Carleson measure if
\[
\left\Vert \mu\right\Vert _{W^{1}(\Omega)}:=\sup_{z\in\partial\Omega,\,0<\varepsilon<\varepsilon_{0}}\frac{\left|\mu\right|\left(P_{\varepsilon}(z)\cap\Omega\right)}{\sigma\left(P_{\varepsilon}(z)\cap\partial\Omega\right)}+\left|\mu\right|\left(\Omega\right)<+\infty,
\]
$\varepsilon_{0}=\alpha\delta_{0}$, for $\alpha$ small enough, where
$P_{\varepsilon}(z)$ is the extremal polydisk defined in the next
section and $\sigma$ the surface measure on $\partial\Omega$. $W^{1}(\Omega)$
will denote the space of Carleson measures on $\Omega$. Then, following
ideas initiated in \cite{Bruna-Charp-Dupain-Annals} and adapted by
W. Alexandre (\cite[Definition 1.2]{Ale17}) to non smooth forms we
consider the following terminology (see \secref{Carleson-currents}
for more details):

A current $\vartheta=\sum_{i,j=1}^{n}\vartheta_{ij}dz_{i}\wedge d\overline{z}_{j}$
of degree $\left(1,1\right)$ and order zero in $\Omega$ is called
a Carleson current (in $\Omega$) if
\[
\left\Vert \vartheta\right\Vert _{W^{1}(\Omega)}:=\sup_{u_{1},u_{2}}\left\Vert \frac{\delta_{\Omega}\left|\vartheta\left(u_{1},u_{2}\right)\right|}{k\left(\cdot,u_{1}\right)k\left(\cdot,u_{2}\right)}\right\Vert _{W^{1}(\Omega)}+\left(\delta_{\Omega}\left|\vartheta\right|\right)(\Omega)<+\infty,
\]
where supremum is taken over all smooth vector fields $u_{1}=\left(u_{1}^{i}\right)_{i}$
and $u_{2}=\left(u_{2}^{i}\right)_{i}$ never vanishing in $\Omega$,
$\left|\vartheta\left(u_{1},u_{2}\right)\right|$ is the absolute
value of the measure $\vartheta\left(u_{1},u_{2}\right)=\sum_{i,j}\vartheta_{ij}u_{1}^{i}\overline{u}_{2}^{j}$
and $k\left(\cdot,u_{k}\right)$ the continuous function $z\rightarrow k\left(z,u_{k}\right)$
defined before.

Similarly, a current $\omega$ of degree $1$ and order zero in $\Omega$
is called a Carleson current (in $\Omega$) if
\[
\left\Vert \omega\right\Vert _{W^{1}(\Omega)}:=\sup_{u}\left\Vert \frac{\left|\omega(u)\right|}{k\left(\cdot,u\right)}\right\Vert _{W^{1}(\Omega)}+\left|\omega\right|(\Omega)<+\infty.
\]

\begin{rem*}
In the above definitions, the expressions are independent of the modulus
of the vector fields. Thus they can always be chosen of modulus one.
\end{rem*}

\begin{mainthm}\textit{\label{mthm:Main-Theorem}Let $\Omega$ be
a smoothly bounded lineally convex domain of finite type in $\mathbb{C}^{n}$.
Let $X$ be a divisor in $\Omega$ and $\vartheta_{X}$ the associated
$\left(1,1\right)$-current of integration. Then, if $\vartheta_{X}$
is a Carleson current and if the cohomology class of $X$ in $H^{2}(\Omega,\mathbb{Z})$
is zero, there exist $p>0$ and $f$ in the Hardy space $H^{p}(\Omega)$
such that $X$ is the zero set of $f$.}\end{mainthm}

The general scheme of the proof is now standard (see \cite{VarHp80,AnCa-Hp-90,BrGrHp99,Ngu01,Ale17}):
following Lelong's theory, we have to find a plurisubharmonic function
$u$ such that $i\partial\overline{\partial}u=\vartheta_{X}$ satisfying
a BMO estimate on $\partial\Omega$, and, as in \cite{VarHp80} (and
\cite{Ngu01}), the conclusion will follow the John-Nirenberg theorem
(\cite{JoNi61}). The two main steps are the resolution of the equation
$idw=\vartheta$ with a Carleson estimate and the resolution of the
$\overline{\partial}_{b}$-equation with a BMO estimate on the boundary
of $\Omega$:
\begin{stthm}
\label{thm:d_resolution}Let $\vartheta$ be a closed Carleson current
of degree $\left(1,1\right)$ and order $0$ in $\Omega$ such that
his canonical cohomology class in $H^{2}(\Omega:\mathbb{C})$ is zero.
Then there exists a Carleson current of degree $1$ and order $0$
$\omega$ satisfying $d\omega=\vartheta$. Furthermore, if $\vartheta$
is real, $\omega$ can be chosen real.\end{stthm}

\begin{rem*}
This theorem could have been stated in the more general context of
geometrically separated domains $\Omega$ introduced in \cite{CD08}.
But, as we cannot prove the Main Theorem in that case, and because
the technical details of the proof would be much more complicated,
we restrict us to the case of lineally convex domains of finite type.
\end{rem*}

The second step is based on the proof of \cite[Theorem 2.4]{ChDu18}:
\begin{stthm}
\label{thm:d-bar_resolution}There exists a constant $C>0$ such that,
for all $\overline{\partial}$-closed Carleson current of degree $\left(0,1\right)$
and order $0$ $\omega$ on $\Omega$, there exists a solution of
the equation $\overline{\partial}_{b}u=\omega$ such that $\left\Vert u\right\Vert _{BMO(\partial\Omega)}\leq C\left\Vert \omega\right\Vert _{W^{1}(\Omega)}$.
\end{stthm}

Note that in \cite{ChDu18} this last result is stated, for smooth
forms, with $\left\Vert \left\Vert \omega(\zeta)\right\Vert _{\mathbbmss k}d\lambda\right\Vert _{W^{1}}$
instead of $\left\Vert \omega\right\Vert _{W^{1}(\Omega)}$, but we
will see in \secref{Carleson-currents} that, for smooth forms, these
two quantities are equivalent.

\bigskip{}

\thmref{d_resolution} is proved in \secref{Proof-of-theorem-2.1}:
after a regularization procedure we will essentially follow the general
scheme developed in \cite{AnCa-Hp-90} (see also \cite{Sko76} and
\cite{VarHp80}), the technical part being a strong modification of
the calculus made in \cite{BrGrHp99}.

\thmref{d-bar_resolution} is proved in \secref{Proof-of-d-bar-thm}:
once again, after a convenient regularization we use the methods developed
in \cite{ChDu18} and in \cite{Sko76}.

\section{Geometry of lineally convex domains of finite type\label{sec:Geometry-of-lineally}}

The anisotropic geometry of lineally convex domains of finite type
is described in \cite{Conrad_lineally_convex}. Let us just recall
the basic estimates (from \cite{ChDu18}) we will use in the next
section.

For $\zeta$ close to $\partial\Omega$ and $\varepsilon\leq\varepsilon_{0}$,
$\varepsilon_{0}$ small, define, for all non zero vector $v$,
\begin{equation}
\tau\left(\zeta,v,\varepsilon\right)=\sup\left\{ c\mbox{ such that }\rho\left(\zeta+\lambda v\right)-\rho(\zeta)<\varepsilon,\,\forall\lambda\in\mathbb{C},\,\left|\lambda\right|<c\right\} .\label{eq:definition-tau-general}
\end{equation}
Note that the lineal convexity hypothesis implies that the function
$\left(\zeta,\varepsilon\right)\mapsto\tau(\zeta,v,\varepsilon)$
is smooth. In particular, $\zeta\mapsto\tau(\zeta,v,\delta_{\Omega}(\zeta))$
is a smooth function. The pseudo-balls $\mathcal{B}_{\varepsilon}(\zeta)=\mathcal{B}(\zeta,\varepsilon)$
(for $\zeta$ close to the boundary of $\Omega$) of the homogeneous
space associated to the anisotropic geometry of $\Omega$ are
\begin{equation}
\mathcal{B}_{\varepsilon}(\zeta)=\left\{ \xi=\zeta+\lambda u\mbox{ with }\left|u\right|=1\mbox{ and }\left|\lambda\right|<c_{0}\tau(\zeta,u,\varepsilon)\right\} \label{eq:def-pseudo-balls}
\end{equation}
where $c_{0}$ is chosen sufficiently small depending only on the
defining function $\rho$ of $\Omega$ and we define
\[
d(\zeta,z)=\inf\left\{ \varepsilon\mbox{ such that }z\in\mathcal{B}_{\varepsilon}(\zeta)\right\} .
\]

Let $\zeta$ and $\varepsilon$ be fixed. Then, an orthonormal basis
$\left(v_{1},v_{2},\ldots,v_{n}\right)$ is called \emph{$\left(\zeta,\varepsilon\right)$-extremal}
(or $\varepsilon$-\emph{extremal}, or simply \emph{extremal}) if
$v_{1}$ is the complex normal (to $\rho$) at $\zeta$, and, for
$i>1$, $v_{i}$ belongs to the orthogonal space of the vector space
generated by $\left(v_{1},\ldots,v_{i-1}\right)$ and minimizes $\tau\left(\zeta,v,\varepsilon\right)$
in the unit sphere of that space. In association to an extremal basis,
we denote
\begin{equation}
\tau(\zeta,v_{i},\varepsilon)=\tau_{i}(\zeta,\varepsilon).\label{eq:definition-tau-extremal-basis}
\end{equation}

Then we defined polydiscs $AP_{\varepsilon}(\zeta)$ by
\begin{equation}
AP_{\varepsilon}(\zeta)=\left\{ z=\zeta+\sum_{k=1}^{n}\lambda_{k}v_{k}\mbox{ such that }\left|\lambda_{k}\right|\leq c_{0}A\tau_{k}(\zeta,\varepsilon)\right\} .\label{eq:def-polydisk}
\end{equation}

$P_{\varepsilon}(\zeta)$ being the corresponding polydisc with $A=1$
and we also define
\[
d_{1}(\zeta,z)=\inf\left\{ \varepsilon\mbox{ such that }z\in P_{\varepsilon}(\zeta)\right\} .
\]

\begin{rem*}
Note that there are neither uniqueness of the extremal basis $\left(v_{1},v_{2},\ldots,v_{n}\right)$
nor of associated polydisk $P_{\varepsilon}(\zeta)$. However the
functions $\tau_{i}$ and the polydisks associated to two different
$\left(\zeta,\varepsilon\right)$-extremal basis are equivalent. Thus
in all the paper $P_{\varepsilon}(\zeta)=P(\zeta,\varepsilon)$ will
denote a polydisk associated to any $\left(\zeta,\varepsilon\right)$-extremal
basis and $\tau_{i}(\zeta,\varepsilon)$ the radius of $P_{\varepsilon}(\zeta)$.
\end{rem*}

The fundamental result here is that $d$ and $d_{1}$ are equivalent
pseudo-distance which means that there exists a constant $K$ and,
$\forall\alpha>0$, constants $c(\alpha)$ and $C(\alpha)$ such that
\begin{equation}
\mbox{for }\zeta\in P_{\varepsilon}(z),\,P_{\varepsilon}(z)\subset P_{K\varepsilon}(\zeta),\label{eq:inclusion-polydisques-pseudo}
\end{equation}
and
\begin{equation}
c(\alpha)P_{\varepsilon}(\zeta)\subset P_{\alpha\varepsilon}(\zeta)\subset C(\alpha)P_{\varepsilon}(\zeta)\mbox{ and }P_{c(\alpha)\varepsilon}(\zeta)\subset\alpha P_{\varepsilon}(\zeta)\subset P_{C(\alpha)\varepsilon}(\zeta).\label{eq:polydiscs-pseudodistance}
\end{equation}

Moreover the pseudo-balls $B_{\varepsilon}$ and the polydiscs $P_{\varepsilon}$
are equivalent in the sense that there exists a constant $K>0$ depending
only on $\Omega$ such that
\begin{equation}
\frac{1}{K}P_{\varepsilon}(\zeta)\subset\mathcal{B}_{\varepsilon}(\zeta)\subset KP_{\varepsilon}(\zeta),\label{eq:equivalence-pseudoballs-polydisk}
\end{equation}
so
\[
d(\zeta,z)\simeq d_{1}(\zeta,z).
\]

Let us recall for $\zeta$ close to $\partial\Omega$ and $\varepsilon>0$
small, other basic properties of this geometry (see \cite{Conrad_lineally_convex}
and \cite{CDMb}):
\begin{sllem}
\label{lem:basic-properties-geometry}\mynobreakpar
\begin{enumerate}
\item \label{geometry-1}Let $w=\left(w_{1},\ldots,w_{n}\right)$ be an
orthonormal system of coordinates centered at $\zeta$. Then
\[
\left|\frac{\partial^{\left|\alpha+\beta\right|}\rho(\zeta)}{\partial w^{\alpha}\partial\bar{w}^{\beta}}\right|\lesssim\frac{\varepsilon}{\prod_{i}\tau\left(\zeta,w_{i},\varepsilon\right)^{\alpha_{i}+\beta_{i}}},\,\left|\alpha+\beta\right|\geq1.
\]

\item \label{geometry-3}If $\left(v_{1},\ldots,v_{n}\right)$ is a $\left(\zeta,\varepsilon\right)$-extremal
basis and $\gamma=\sum_{1}^{n}a_{j}v_{j}\neq0$, then
\[
\frac{1}{\tau(\zeta,\gamma,\varepsilon)}\simeq\sum_{j=1}^{n}\frac{\left|a_{j}\right|}{\tau_{j}(\zeta,\varepsilon)}.
\]

\item \label{geometry-4}If $v$ is a unit vector then:

\begin{enumerate}
\item $z=\zeta+\lambda v\in P_{\varepsilon}(\zeta)$ implies $\left|\lambda\right|\lesssim\tau(\zeta,v,\varepsilon)$,
\item $z=\zeta+\lambda v$ with $\left|\lambda\right|\leq\tau(\zeta,v,\varepsilon)$
implies $z\in CP_{\varepsilon}(\zeta)$.
\end{enumerate}
\item \label{geometry-5}If $\nu$ is the unit complex normal vector, then
$\tau(\zeta,v,\varepsilon)=\varepsilon$ and if $v$ is any unit vector
and $\lambda\geq1$,
\begin{equation}
\lambda^{\nicefrac{1}{m}}\tau(\zeta,v,\varepsilon)\lesssim\tau(\zeta,v,\lambda\varepsilon)\lesssim\lambda\tau(\zeta,v,\varepsilon),\label{eq:comp-tau-epsilon-tau-lambda-epsilon}
\end{equation}
where $m$ is the type of $\Omega$.
\end{enumerate}
\end{sllem}

\begin{sllem}[Lemma 3.4 of \cite{ChDu18}]
\label{lem:3.4-maj-deriv-rho-equiv-tho-i-z-zeta}For $z$ close to
$\partial\Omega$, $\varepsilon$ small and $\zeta\in P_{\varepsilon}(z)$
or $z\in P_{\varepsilon}(\zeta)$, we have, for all $1\leq i\leq n$:
\begin{enumerate}
\item $\tau_{i}(z,\varepsilon)=\tau\left(z,v_{i}\left(z,\varepsilon\right),\varepsilon\right)\simeq\tau\left(\zeta,v_{i}\left(z,\varepsilon\right),\varepsilon\right)$
where $\left(v_{i}\left(z,\varepsilon\right)\right)_{i}$ is the $\left(z,\varepsilon\right)$-extremal
basis;
\item $\tau_{i}(\zeta,\varepsilon)\simeq\tau_{i}(z,\varepsilon)$;
\item In the coordinate system $\left(z_{i}\right)$ associated to the $\left(z,\varepsilon\right)$-extremal
basis, $\left|\frac{\partial\rho}{\partial z_{i}}(\zeta)\right|\lesssim\frac{\varepsilon}{\tau_{i}}$
where $\tau_{i}$ is either $\tau_{i}\left(z,\varepsilon\right)$
or $\tau_{i}\left(\zeta,\varepsilon\right)$.
\end{enumerate}
\end{sllem}

\begin{srrem}
Clearly, for $\delta_{\Omega}(z)\leq\delta_{1}$ and all non zero
vector $v$, we can extend smoothly the functions $\tau\left(z,v,\varepsilon\right)$
to all $\varepsilon$ and we can also define vectors $e_{i}(z,\varepsilon)$
and polydisks $P(z,\varepsilon)$, so that the above properties remain
true with constants depending on $A$ for $\varepsilon\mbox{ and }\lambda\varepsilon\in\left]0,A\right]$,
$\delta_{\Omega}(z)\mbox{ and }\delta_{\Omega}(\zeta)\leq\delta_{1}$.

Of course the new $\left(e_{i}(z,\varepsilon)\right)_{i}$ are not
extremal basis in the original sense but we will call them again extremal
basis.
\end{srrem}

\section{Some properties of Carleson currents\label{sec:Carleson-currents}}

In the previous section, we defined the terminology of Carleson current
of degree $\left(1,1\right)$ or $1$. We extend it to general currents
$T$ of degree $2$ with the same definition:

Let $T=\sum_{i<j}T_{i,j}^{0}dz^{i}\wedge dz^{j}+\sum T_{i,j}^{1}dz^{i}\wedge d\overline{z}^{j}+\sum_{i<j}T_{i,j}^{2}d\overline{z}^{i}\wedge d\overline{z}^{j}$
then
\[
\left\Vert T\right\Vert _{W^{1}(\Omega)}=\sup_{u_{1},u_{2}}\left\Vert \frac{\delta_{\Omega}\left|T\left(u_{1},u_{2}\right)\right|}{k\left(\cdot,u_{1}\right)k\left(\cdot,u_{2}\right)}\right\Vert _{W^{1}(\Omega)}+\left(\delta_{\Omega}\left|T\right|\right)(\Omega)<+\infty,
\]
where supremum is taken over all smooth vector fields $u_{1}=\left(u_{1}^{i}\right)_{i}$
and $u_{2}=\left(u_{2}^{i}\right)_{i}$ never vanishing in $\Omega$,
and $\left|T\left(u_{1},u_{2}\right)\right|$ is the absolute value
of the measure
\[
T\left(u_{1},u_{2}\right)=\sum_{i<j}T_{i,j}^{0}u_{1}^{i}u_{2}^{j}+\sum T_{i,j}^{1}u_{1}^{i}\overline{u_{2}^{j}}+\sum_{i<j}T_{i,j}^{2}\overline{u_{1}^{i}}\overline{u_{2}^{j}}.
\]

Moreover, let $V$ be an open set in $\Omega$ and $T$ a current
of degree $1$ or $2$ and order zero in $V$. We say that $T$ is
a Carleson current in $\Omega$ if the current $\chi_{V}T$, where
$\chi_{V}$ is the characteristic function of $V$, is a Carleson
current in $\Omega$ and we denote $\left\Vert T\right\Vert _{W^{1}(\Omega)}:=\left\Vert \chi_{V}T\right\Vert _{W^{1}(\Omega)}$.

Note that, if $V$ is relatively compact in $\Omega$, a current $T$
in $V$ is a Carleson current (in $\Omega$) if the coefficients of
$T$ are bounded measures.

\bigskip{}

In the two next sections we need to regularize Carleson currents to
be able to write explicit formulas solving the $d$ or the $\overline{\partial}$
equation. This is done classically using convolutions (see M. Andersson
and H. Carlsson, \cite[page 472]{AnCa-Hp-90} in the case of strictly
pseudo-convex domains), and, because of the definition of the $W^{1}(\Omega)$
norm for currents, we give below some details (for currents of degree
$\left(1,1\right)$ to simplify notations) when $V$ is contained
in a small neighborhood of a point of $\partial\Omega$ ($V\subset\left\{ \delta_{\Omega}(z)<\beta\delta_{1}\right\} $).

For $\varepsilon>0$ sufficiently small, let $\varphi_{\epsilon}=\frac{1}{\varepsilon^{2n}}\varphi\left(\frac{z}{\varepsilon}\right)$
where $\varphi$ is a $\mathcal{C}^{\infty}$-smooth non negative
function supported in the ball $\left\{ \left|z\right|<\nicefrac{1}{2}\right\} $
of $\mathbb{C}^{n}$ such that $\int\varphi=1$. Let $T=\sum T_{I,J}dz^{I}\wedge d\overline{z}^{J}$
be a Carleson current of order zero in an open set $V$ of $\Omega$.
Let 
\[
V_{\varepsilon}=\left\{ z\in V\mbox{ such that }\delta_{\partial V}(z)>\varepsilon\right\} .
\]
 Then for $z\in V_{\varepsilon}$ define
\[
T_{\varepsilon}=\sum_{I,J}T_{I,J}*\varphi_{\varepsilon}dz^{I}\wedge d\overline{z}^{J}
\]
so that $T_{\varepsilon}$ is a smooth form in $V_{\varepsilon}$.
\begin{spprop}
\label{prop:regularization-local}With the above notations, if $T$
is a closed Carleson current of degree $2$ or $1$ in $V$, then
the forms $T_{\varepsilon}$ are closed and $\left\Vert T_{\varepsilon}\right\Vert _{W^{1}(\Omega)}\lesssim\left\Vert T\right\Vert _{W^{1}(\Omega)}$.\end{spprop}

\begin{proof}
To simplify the notations, we make the proof for $T$ of degree $\left(1,1\right)$,
that is $T=\sum_{i,j}T_{i,j}dz_{i}\wedge d\overline{z}_{j}$ and $T_{\varepsilon}=\sum_{i,j}T_{i,j}^{\varepsilon}dz_{i}\wedge d\overline{z}_{j}$
with $T_{i,j}^{\varepsilon}(z)=\int_{\left\{ \left|z-\zeta\right|<\nicefrac{1}{2}\right\} }\varphi_{\varepsilon}(\zeta-z)dT_{i,j}(\zeta)$.

Let $Z\in\partial\Omega$. If $t<c\varepsilon$ ($c$ small enough
depending only on $\Omega$) then $B(Z,t)\cap V_{\varepsilon}=\emptyset$.
Let us assume $c\varepsilon<t\leq\varepsilon_{0}$. We have to estimate
\begin{eqnarray*}
I & = & \int\chi(z)\frac{\delta_{\Omega}(z)}{k(z,u)k(z,v)}\left|\sum_{i,j}T_{i,j}^{\varepsilon}(z)u_{i}(z)\overline{v}_{j}(z)\right|d\lambda(z)\\
 & = & \sup_{\left|f\right|\leq1}\int f(z)\chi(z)\frac{\delta_{\Omega}(z)}{k(z,u(z))k(z,v(z))}\left(\sum_{i,j}T_{i,j}^{\varepsilon}(z)u_{i}(z)\overline{v}_{j}(z)\right)d\lambda(z)
\end{eqnarray*}
where $\chi$ is the characteristic function of $B(Z,t)\cap V_{\varepsilon}$
and $u$ and $v$ are smooth vector fields never vanishing in $\Omega$.
Using the definition of $T_{\varepsilon}$ we get
\[
I=\sup_{\left|f\right|\leq1}\int\sum_{i,j}\left(\int f(z)\chi(z)\frac{\delta_{\Omega}(z)u_{i}(z)\overline{v}_{j}(z)}{k(z,u(z))k(z,v(z))}\varphi_{\varepsilon}(\zeta-z)d\lambda(z)\right)dT_{i,j}(\zeta).
\]

Note now that the function 
\[
z\mapsto f(z)\chi(z)\frac{\delta_{\Omega}(z)u_{i}(z)\overline{v}_{j}(z)}{k(z,u(z))k(z,v(z))}\varphi_{\varepsilon}(\zeta-z)
\]
 is supported in $B(Z,t)\cap\left\{ \delta_{\Omega}>\nicefrac{\varepsilon}{2}\right\} $.
Moreover, for $z\in V_{\varepsilon}$ and $\left|\zeta-z\right|<\nicefrac{\varepsilon}{2}$,
$\delta_{\Omega}(z)\simeq\delta_{\Omega}(\zeta)$, and, $\zeta\in P\left(z,K_{1}\delta_{\Omega}(z)\right)$.
Then by (1) of \lemref{3.4-maj-deriv-rho-equiv-tho-i-z-zeta} and
(4) of \lemref{basic-properties-geometry}, $k(z,u(z))\simeq k(\zeta,u(z))$
and $k(z,v(z))\simeq k(\zeta,v(z))$, so
\[
I\lesssim\sup_{\left|f\right|\leq1}\int\sum_{i,j}\left(\int f(z)\chi(z)\frac{\delta_{\Omega}(\zeta)u_{i}(z)\overline{v}_{j}(z)}{k(\zeta,u(z))k(\zeta,v(z))}\varphi_{\varepsilon}(\zeta-z)d\lambda(z)\right)dT_{i,j}(\zeta).
\]

Making the change of variables $\zeta-z=\xi$ and applying Fubini
theorem, we get ( $\chi(\zeta-\xi)\neq0$ implies $\zeta\in P(Z,Kt)$)
\begin{eqnarray*}
I & \lesssim & \int\varphi_{\varepsilon}(\xi)\left[\sup_{\left|f\right|\leq1}\int f(\zeta-\xi)\chi(\zeta-\xi)\frac{\delta_{\Omega}(\zeta)}{k(\zeta,u(\zeta-\xi))k(\zeta,v(\zeta-\xi))}\right.\\
 &  & \hspace{5.3cm}\left.\sum_{i,j}u_{i}(\zeta-\xi)\overline{v}_{j}(\zeta-\xi)dT_{i,j}(\zeta)\right]d\lambda(\xi)\\
 & = & \int\varphi_{\varepsilon}(\xi)\left[\int_{B(Z,Kt)}\chi(\zeta-\xi)\frac{\delta_{\Omega}(\zeta)}{k(\zeta,u(\zeta-\xi))k(\zeta,v(\zeta-\xi))}\right.\\
 &  & \hspace{5.7cm}\left.d\left|\sum_{i,j}\tau_{\xi}\left(u_{i}\right)\tau_{\xi}\left(\overline{v}_{j}\right)T_{i,j}\right|(\zeta)\right]d\lambda(\xi).
\end{eqnarray*}
where $\tau_{\xi}\left(u_{i}\right)(\zeta)=u_{i}(\zeta-\xi)$, $\tau_{\xi}\left(v_{i}\right)(\zeta)=v_{i}(\zeta-\xi)$.

Finally, (\ref{eq:polydiscs-pseudodistance}) gives, denoting $u_{\xi}(\zeta)=u(\zeta-\xi)$
and $v_{\xi}(\zeta)=v(\zeta-\xi)$ 
\[
I\lesssim\sup_{\left|\xi\right|<\nicefrac{\varepsilon}{2}}\int_{B(Z,Kt)\cap\left\{ \delta_{\Omega}>\nicefrac{\varepsilon}{2}\right\} }\frac{\delta_{\Omega}(\zeta)}{k(\zeta,u_{\xi}(\zeta))k(\zeta,v_{\xi}(\zeta))}d\left|T(u_{\xi},v_{\xi})\right|
\]
which concludes the proof, as the smooth vector fields $u_{\xi}$
and $v_{\xi}$ can be viewed as smooth vector fields in $\Omega$
never vanishing, and, for $\beta$ small enough, $V$ is contained
in the union of the tents.
\end{proof}
If $T$ is globally defined in $\Omega$, then the forms $T_{\varepsilon}$
are defined in $\left\{ z\in\Omega\mbox{ such that }\delta_{\Omega}(z)>\varepsilon\right\} $
so there exists a constant $C$ (depending only on $\rho$) such that
they are defined in $\Omega^{\varepsilon}=\left\{ \rho<-C\varepsilon\right\} $,
$\varepsilon$ small enough, but \emph{they are not} Carleson currents
in $\Omega^{\varepsilon}$ in general. Then to be able to use this
regularization procedure in the last section we have to introduce
a notion of $s$-Carleson current.

Let $s>0$ small. We say that a measure $\mu$ in $\Omega$ is a $s$-Carleson
measure if
\[
\left\Vert \mu\right\Vert _{W_{s}^{1}(\Omega)}:=\sup_{z\in\partial\Omega,\,s<\varepsilon<\varepsilon_{0}}\frac{\left|\mu\right|\left(P_{\varepsilon}(z)\cap\Omega\right)}{\sigma\left(P_{\varepsilon}(z)\cap\partial\Omega\right)}+\left|\mu\right|\left(\Omega\right)<+\infty,
\]
and we say that a $1$-current $\omega$ of order zero is a $s$-Carleson
current in $\Omega$ if
\[
\left\Vert \omega\right\Vert _{W_{s}^{1}(\Omega)}:=\sup_{u}\left\Vert \frac{\left|\omega(u)\right|}{k\left(\cdot,u\right)}\right\Vert _{W_{s}^{1}(\Omega)}+\left|\omega\right|(\Omega)<+\infty.
\]

Then:
\begin{spprop}
\label{prop:regularization-global}There exists a constant $C$ depending
only on $\rho$ such that, if $T$ is a closed Carleson current in
$\Omega$ then, for $\varepsilon$ small, the closed forms $T_{\varepsilon}$
are $\varepsilon$-Carleson currents in $\Omega^{\varepsilon}=\left\{ \rho<-C\varepsilon\right\} $
and $\left\Vert T_{\varepsilon}\right\Vert _{W_{\varepsilon}^{1}(\Omega^{\varepsilon})}\lesssim\left\Vert T\right\Vert _{W^{1}(\Omega)}$.\end{spprop}

\begin{proof}
By \propref{regularization-local} it suffices to show that $\left\Vert T_{\varepsilon}\right\Vert _{W_{\varepsilon}^{1}\left(\Omega^{\varepsilon}\right)}\lesssim\left\Vert T_{\varepsilon}\right\Vert _{W^{1}(\Omega)}$.
Let $z\in\partial\Omega^{\varepsilon}$ and let $Z$ be the projection
of $z$ on $\partial\Omega$. For $t>\varepsilon$, $P(z,t)\subset P(Z,Kt)$
and the proposition follows (\ref{eq:comp-tau-epsilon-tau-lambda-epsilon})
and \lemref{3.4-maj-deriv-rho-equiv-tho-i-z-zeta}.
\end{proof}

Finally, to solve the $d$ and $\overline{\partial}$ equations with
good estimates, we need to compare the notion of Carleson current
for smooth currents of degree $2$ or $1$ $T$ with a convenient
punctual norm $\left\Vert \left\Vert T(\zeta)\right\Vert _{\mathbbmss k}d\lambda\right\Vert _{W^{1}(\Omega)}$:
we define
\[
\left\Vert T(\zeta)\right\Vert _{\mathbbmss k}=\sup_{v_{i}\in\mathbb{C}^{n},\,\left\Vert v_{i}\right\Vert =1}\frac{\left|T\left(v_{1},v_{2}\right)(\zeta)\right|}{k\left(\zeta,v_{1}\right)k\left(\zeta,v_{2}\right)},
\]
for forms of degree $2$ and
\[
\left\Vert T(\zeta)\right\Vert _{\mathbbmss k}=\sup_{\left\Vert v\right\Vert =1}\frac{\left|T(v)(\zeta)\right|}{k\left(\zeta,v\right)},
\]
 for forms of degree $\left(0,1\right)$.

If there were smooth vector fields $\left(\widetilde{e}_{i}\right)_{1\leq i\leq n}$
such that, at each point $z$, $\left(\widetilde{e}_{i}(z)\right)_{i}$
is a $\left(z,\delta_{\Omega}(z)\right)$-extremal basis, this comparison
would be immediate and, as noted by several authors, many points of
the theory of convex (and lineally convex) domains of finite type
would be simplified. Unfortunately this is not the case, and, in the
case of convex domains of finite type and smooth currents, W. Alexandre
overcomes this difficulty using a base of the Bergman metric (and
estimates of this metric proved by J. Mc Neal, see \cite[Proposition 2.12]{Ale17}).
The same result could be proved in our context of lineally convex
domains using the results of \cite{CD08}. However, we do not use
this method because it is quite easy to show, in general, that the
$W^{1}(\Omega)$-norm of a current is controlled by vector fields
``almost extremal'':
\begin{spprop}
\label{prop:comparison-carleson-norms-smooth-forms}Let $\psi$ be
a current of order zero of degree $2$ or $1$ in an open set $U\subset\left\{ \delta_{\Omega}(z)<\beta\delta_{1}\right\} $
of $\Omega$.
\begin{enumerate}
\item There exist $n$ smooth vector fields $u_{i}$ never vanishing in
$\Omega$ such that, if $\psi$ is of degree $2$, 
\[
\left\Vert \psi\right\Vert _{W^{1}(\Omega)}\simeq\sum_{i,j}\left\Vert \frac{\delta_{\Omega}\left|\psi\left(u_{i},u_{j}\right)\right|}{k(\cdot,u_{i})k(\cdot,u_{j})}\right\Vert _{W^{1}(\Omega)},
\]
the vector fields $u_{i}$ coinciding, outside a set of $\left|\psi\right|$-measure
arbitrary small, with extremal basis in the sense of geometrically
separated domains, and 
\[
\left\Vert \psi\right\Vert _{W^{1}(\Omega)}\simeq\sum_{i}\left\Vert \frac{\left|\psi(u_{i})\right|}{k(\cdot,u_{i})}\right\Vert _{W^{1}(\Omega)}
\]
if it is of degree $1$, the constants in the equivalence being independent
of $\psi$.
\item Moreover, if $\psi$ is smooth. Then:

\begin{enumerate}
\item $\left\Vert \psi\right\Vert _{W^{1}(\Omega)}\simeq\left\Vert \delta_{\Omega}\left\Vert \chi_{U}\psi\right\Vert _{\mathbbmss k}d\lambda\right\Vert _{W^{1}(\Omega)}$
if $\psi$ is of degree $2$, and $\left\Vert \psi\right\Vert _{W^{1}(\Omega)}\simeq\left\Vert \left\Vert \chi_{U}\psi\right\Vert _{\mathbbmss k}d\lambda\right\Vert _{W^{1}(\Omega)}$
if not, the constants in the equivalence being independent of $\psi$;
\item for $s>0$, if $\psi$ is of degree $1$, $\left\Vert \psi\right\Vert _{W_{s}^{1}(\Omega)}\simeq\left\Vert \left\Vert \chi_{U}\psi\right\Vert _{\mathbbmss k}d\lambda\right\Vert _{W_{s}^{1}(\Omega)}$
the constants in the equivalence being independent of $\psi$.
\end{enumerate}
\end{enumerate}
\end{spprop}

\begin{rem*}
\mynobreakpar
\begin{enumerate}
\item In \cite[Proposition 2.12]{Ale17} W. Alexandre proved (2) of the
proposition for convex domains of finite type, using \emph{universal}
vector fields (i.e. depending only on $\Omega$ but not on $\psi$)
related to the Bergman metric.
\item The extremal basis in the sense of geometrically separated domains
(\cite{CD08}) are not stricto-sensus extremal basis in J. Mc. Neal
and M. Conrad sense, but they give the same homogeneous space.
\item The equivalences of (2) of the proposition can be proved directly
without using any extremal basis, using simply the continuity of the
functions $z\mapsto\frac{\delta_{\Omega}(z)\psi(u,v)(z)}{k(z,u(z))k(z,v(z))}$
which gives an equivalent ot Lemma 1 below on small euclidean balls.
The final construction of the vectors fields $\left(u_{i}\right)_{i}$
is analog (and easyer).
\item Even if the proof of \propref{Local-solution-d-equation} needs only
the second part of the proposition, we thought that it is interesting
to present the assertion in the general case of non smooth currents.
\end{enumerate}
\end{rem*}

\begin{proof}[Proof of \propref{comparison-carleson-norms-smooth-forms}]
We only do the proof for currents $\psi$ of degree $2$. The inequality
$\gtrsim$ is trivial, so we prove the converse one.
\begin{llemsp}
Let $w\in U$ and $\left(e_{i}\right)_{i}$ be a $\delta_{\Omega}(w)$-extremal
basis at $w$. Let $u$ and $v$ be two smooth non vanishing vectors
fields.
\begin{enumerate}
\item Assume the coefficients of $\psi$ are measures. Then, for every measurable
set $D\subset P\left(w,\delta_{\Omega}(w)\right)\cap U$, we have
\begin{equation}
\int_{D}\frac{\delta_{\Omega}(\xi)d\left|\psi\left(u,v\right)\right|(\xi)}{k(\xi,u(\xi))k(\xi,v(\xi))}\lesssim\sum_{i,j}\int_{D}\frac{\delta_{\Omega}(\xi)d\left|\psi\left(e_{i},e_{j}\right)\right|(\xi)}{k\left(\xi,e_{i}\right)k\left(\xi,e_{j}\right)}.\label{eq:lemma1-comparison-non-smooth-forms}
\end{equation}

\item Moreover, if $\psi$ smooth. Then, for $\xi\in P\left(w,\delta_{\Omega}\right)\cap U$
we have
\begin{equation}
\frac{\left|\psi\left(u,v\right)(\xi)\right|}{k(\xi,u(\xi))k(\xi,v(\xi))}\lesssim\sum_{i,j}\frac{\left|\psi\left(e_{i},e_{j}\right)(\xi)\right|}{k\left(\xi,e_{i}\right)k\left(\xi,e_{j}\right)}.\label{eq:lemma1-comparison-smooth-forms}
\end{equation}

\end{enumerate}
\end{llemsp}

\begin{proof}
Decomposing $u$ and $v$ on the basis $\left(e_{i}\right)_{i}$,
we get
\[
\frac{\left|\psi\left(u,v\right)(\xi)\right|}{k(\xi,u(\xi))k(\xi,v(\xi))}\lesssim\frac{\sum\left|\psi\left(e_{i},e_{j}\right)(\xi)\right|\left|u_{i}(\xi)\right|\left|v_{j}(\xi)\right|}{k(\xi,u(\xi))k(\xi,v(\xi))},
\]
if $\psi$ is smooth, and, if not
\[
\int_{D}\frac{\delta_{\Omega}(\xi)d\left|\psi\left(u,v\right)\right|(\xi)}{k(\xi,u(\xi))k(\xi,v(\xi))}\lesssim\sum_{i,j}\int_{D}\frac{\delta_{\Omega}(\xi)\left|u_{i}(\xi)\right|\left|v_{j}(\xi)\right|d\left|\psi\left(e_{i},e_{j}\right)\right|(\xi)}{k(\xi,u(\xi))k(\xi,v(\xi))}.
\]

Now, by (4) of \lemref{basic-properties-geometry} and (2) of \lemref{3.4-maj-deriv-rho-equiv-tho-i-z-zeta},
\begin{eqnarray*}
k(\xi,u(\xi)) & \simeq & \frac{\delta_{\Omega}(\xi)}{\tau\left(\xi,u(\xi),\delta_{\Omega}(w)\right)}\\
 & \simeq & \frac{\delta_{\Omega}(\xi)}{\tau\left(w,u(\xi),\delta_{\Omega}(w)\right)}
\end{eqnarray*}
and, by (2) of \lemref{basic-properties-geometry}, 
\[
\frac{1}{\tau\left(w,u(\xi),\delta_{\Omega}(w)\right)}\simeq\max\frac{\left|u_{i}(\xi)\right|}{\tau_{i}\left(w,\delta_{\Omega}(w)\right)}\mbox{ and }\frac{1}{k(\xi,u(\xi))}\lesssim\min\frac{\tau_{i}\left(w,\delta_{\Omega}(w)\right)}{\left|u_{i}(\xi)\right|}\delta_{\Omega}(\xi)^{-1}.
\]

Then
\begin{eqnarray*}
\frac{\left|u_{i}(\xi)\right|\left|v_{j}(\xi)\right|}{k(\xi,u(\xi))k(\xi,v(\xi))} & \lesssim & \tau_{i}\left(w,\delta_{\Omega}(w)\right)\tau_{j}\left(w,\delta_{\Omega}(w)\right)\delta_{\Omega}(\xi)^{-2}\\
 & \lesssim & \frac{1}{k\left(\xi,e_{i}\right)k\left(\xi,e_{j}\right)}
\end{eqnarray*}
because $\tau_{i}\left(w,\delta_{\Omega}(w)\right)=\tau\left(w,e_{i},\delta_{\Omega}(w)\right)\simeq\tau\left(\xi,e_{i},\delta_{\Omega}(\xi)\right)$
and the lemma is proved.\end{proof}
\begin{llemsp}
Under the conditions of the previous lemma, the inequalities (\ref{eq:lemma1-comparison-smooth-forms})
and (\ref{eq:lemma1-comparison-non-smooth-forms}) are still true
replacing the basis $\left(e_{i}\right)_{i}$ by the basis $\left(e'_{i}\right)_{i}$
where $e'_{i}=e_{i}+\mbox{O}\left(\delta_{\Omega}^{2}(w)\right)$.\end{llemsp}

\begin{proof}
By (2) of \lemref{basic-properties-geometry}, $k\left(\xi,e'_{i}\right)\simeq k\left(\xi,e_{i}\right)$,
and 
\[
\left\langle \psi;e'_{i},e'_{j}\right\rangle =\left\langle \psi;e_{i},e_{j}\right\rangle +\sum_{s,t}\mbox{O}\left(\delta_{\Omega}^{2}(w)\right)\left\langle \psi;e_{s},e_{t}\right\rangle 
\]
which proves the result for $\beta$ small enough.
\end{proof}
We now finish the proof of the \propref{comparison-carleson-norms-smooth-forms},
proving both parts at the same time. Let $P_{i}=P\left(Z_{i},\frac{\delta_{\Omega}\left(Z_{i}\right)}{K}\right)$,
$i\in\mathbb{N}$, be a minimal covering of $U\cap\Omega$. For each
$i$ fixed and $N_{i}$ to be precised later, let
\[
A_{i}^{j}=\left\{ z\in U\cap\Omega\mbox{ such that }d_{e}\left(z,P_{i}\right)<\frac{\delta_{\Omega}\left(Z_{i}\right)}{K}\frac{j}{N_{i}}\right\} ,\,j=1,\ldots,N_{i},
\]
and $A_{i}^{0}=P_{i}$. We assume that $K$ is chosen so that, for
all $j$, $A_{i}^{j}\subset P\left(Z_{i},\delta_{\Omega}\left(Z_{i}\right)\right)$.
Let $B_{i}^{j}=A_{i}^{j}\setminus A_{i}^{j-1}$.

Let $I_{k}=\left\{ i\in\mathbb{N}\mbox{ such that }\delta_{\Omega}\left(Z_{i}\right)\in\left[2^{-k},2^{-k+1}\right[\right\} $
and $M_{k}=\#I_{k}$ the cardinal of $I_{k}$. For each $i\in I_{k}$
let us choose $N_{i}=N(k)$ sufficiently large so that there exists
$s(i)\geq1$ such that
\[
\left|\psi\right|\left(B_{i}^{s(i)}\right)\leq\frac{1}{2^{k+4}}\frac{1}{M_{k}}\left\Vert \psi\right\Vert _{W^{1}(\Omega)}.
\]

Let $C_{i}=A_{i}^{s(i)-1}$. Note that $\left(C_{i}\right)_{i}$ is
an open covering of $U\cap\Omega$. Let $\Delta=\bigcup_{i}B_{i}^{s(i)}$
and let $D_{j}$ the connected components of $\left(U\cap\Omega\right)\setminus\Delta$.

Let $J(j)=\left\{ t\mbox{ such that }D_{j}\subset C_{t}\right\} $.
Let $w_{j}$ one of the points $Z_{t}$ such that $t\in I(j)$. Let
$\left(\psi_{j}\right)_{j}$ be a family of smooth functions such
that $0\leq\psi_{j}\leq1$, $\psi_{j}\equiv1$ on $D_{j}$, $\mbox{Supp}\left(\psi_{j}\right)\subset D_{j}\cup\left\{ A_{t}^{s(t)}\mbox{ such that }t\in J(j)\right\} $,
and $\sum_{j}\psi_{j}\equiv1$ on $U\cap\Omega$.

For each $j$ let $\left(e_{l}^{j}\right)_{l}$ be a $\delta_{\Omega}\left(w_{j}\right)$-extremal
basis at $w_{j}$. Note that we can chose $e_{l}^{j}$ so that the
component of $e_{l}^{j}$ on the first vector of the canonical basis
$\left(f_{k}\right)_{k}$ of $\mathbb{C}^{n}$ is non negative. If
we denote $v_{l}^{j}=e_{l}^{j}+\delta_{\Omega}^{2}\left(w_{j}\right)f_{1}$
and $u_{l}=\sum_{j}\psi_{j}v_{l}^{j}$, the vector fields $u_{l}$,
$1\leq l\leq n$, are smooth, don't vanish on $U$, and the proposition
follows the lemmas.
\end{proof}

\section{Proof of Theorem \ref{thm:d_resolution}\label{sec:Proof-of-theorem-2.1}}

The main point in the proof is the following local version of the
theorem:
\begin{spprop}
\label{prop:Local-solution-d-equation}For each point $p\in\overline{\Omega}$
there exist two neighborhoods $W$ and $V$ of $p$ in $\overline{\Omega}$,
$W\Subset V$ (in $\overline{\Omega}$) such that:
\begin{enumerate}
\item If $\vartheta$is a closed current of order $0$ and degree $2$ supported
in $V\cap\Omega$ such that $\vartheta$ is a Carleson current in
$\Omega$, there exists $w$ a solution of the equation $dw=\vartheta$
in $W$ such that $w$ is a Carleson current in $\Omega$.
\item If $\omega$ is a closed current of order $0$ and degree $1$ supported
in $V\cap\Omega$ such that $\omega$ is a Carleson current in $\Omega$,
there exists $f$ a solution of the equation $df=\omega$ in $W$
such that $\delta_{\Omega}^{\nicefrac{1}{m}-1}f$ is a Carleson measure
in $\Omega$.
\end{enumerate}
\end{spprop}

\medskip{}

For convenience of the reader, let us briefly indicate how \thmref{d_resolution}
is a simple consequence of the Proposition (this follows \cite{Sko76},
\cite{VarHp80}, and \cite{AnCa-Hp-90}).

We consider the following three sheaves $\mathcal{F}_{0}$, $\mathcal{F}_{1}$
and $\mathcal{F}_{2}$:

Let $U$ be an open set in $\overline{\Omega}$. $\vartheta\in\Gamma\left(U,\mathcal{F}_{2}\right)$
if $\vartheta$ is a closed $2$-current supported in $U\cap\Omega$
and $\chi\vartheta$ is a $2$-Carleson current supported in $U\cap\Omega$
for all $\chi\in\mathcal{C}_{0}^{\infty}(U)$; $w\in\Gamma\left(U,\mathcal{F}_{1}\right)$
if $w$ is a $1$-current supported in $U\cap\Omega$, $dw\in\Gamma\left(U,\mathcal{F}_{2}\right)$
and $\chi w$ is a $1$-Carleson current supported in $U\cap\Omega$
for all $\chi\in\mathcal{C}_{0}^{\infty}(U)$; $f\in\Gamma\left(U,\mathcal{F}_{0}\right)$
if $f$ is a measure in $U\cap\Omega$, $df\in\Gamma\left(U,\mathcal{F}_{1}\right)$
and $\chi\delta_{\Omega}^{\nicefrac{1}{m}-1}f$ is a Carleson measure
supported in $U\cap\Omega$ for all $\chi\in\mathcal{C}_{0}^{\infty}(U)$.
Then $\mathcal{F}_{0}$ and $\mathcal{F}_{1}$ are fine sheaves and,
by \propref{Local-solution-d-equation}, the sequence
\[
0\rightarrow\mathbb{C}\stackrel{i}{\rightarrow}\mathcal{F}_{0}\stackrel{d}{\rightarrow}\mathcal{F}_{1}\stackrel{d}{\rightarrow}\mathcal{F}_{2}\rightarrow0
\]
is exact, and, by a standard cohomology argument,
\[
\Gamma\left(\overline{\Omega},\mathcal{F}_{2}\right)/d\Gamma\left(\overline{\Omega},\mathcal{F}_{1}\right)\simeq H^{2}\left(\overline{\Omega},\mathbb{C}\right)\simeq H^{2}\left(\Omega,\mathbb{C}\right)
\]
and \thmref{d_resolution} is proved.

\medskip{}

For $p\in\Omega$, choosing $V$ and $W$ to be euclidean balls relatively
compact in $\Omega$, the proposition only means that if $\vartheta$
(resp $\omega$) is a current whose coefficients are bounded measures
in $V$ there exists $w$ (resp. $f$) solution of $dw=\vartheta$
(resp. $df=\omega$) in $W$ whose coefficients are bounded measures
in $W$. As this is standard, we don't give any details here.

We now prove \propref{Local-solution-d-equation} for a fixed point
$p\in\partial\Omega$.

There exists a strictly positive real number $\delta_{1}$ such that,
for $z\in\Omega$ satisfying $\delta_{\Omega}(z)\leq\delta_{1}$ the
polydisk $P(z,\varepsilon)$ are well defined. Then we choose the
neighborhoods $V$ and $W$ of $p$ as follows: let $r_{1}$, $r_{2}$
and $\eta_{1}$ three positive real numbers, $\delta_{1}>r_{1}>4r_{2}>8\eta_{1}$,
such that (denoting by $B_{e}$ an euclidean ball):
\begin{itemize}
\item $B_{e}\left(p,r_{1}\right)\cap\Omega\subset\left\{ z\in\Omega\mbox{ such that }\delta_{\Omega}(z)<\delta_{1}\right\} $;
\item there exists a point $A(p)\in B_{e}\left(p,r_{1}\right)\cap\left\{ \xi\mbox{ such that }2r_{2}<\delta_{\Omega}(\xi)<\nicefrac{r_{1}}{2}\right\} \cap\Omega$;
\item for $\zeta\in B_{e}\left(p,r_{2}\right)$, $P\left(\zeta,\eta_{1}\right)\cap\Omega\subset B_{e}\left(p,2r_{2}\right)$;
\item for $\zeta\in B_{e}\left(p,2r_{2}\right)$ and all $\xi\in B_{e}\left(p,2r_{2}\right)$,
if $\nu(\xi)=\nabla\rho(\xi)$ is the normal at $\xi$, $\left|\left\langle \nu(\xi),\vv{A(p)\zeta}\right\rangle \right|\geq\frac{1}{2}\left\Vert \vv{A(p)\zeta}\right\Vert \left\Vert \nu(\xi)\right\Vert $.
\end{itemize}

Then we define $V=B_{e}\left(p,r_{1}\right)\cap\Omega$ and $W=B_{e}\left(p,r_{2}\right)\cap\left\{ \zeta\in\Omega\mbox{ such that }\delta_{\Omega}(\zeta)<\eta_{1}\right\} $.

First, we regularize the currents using \propref{regularization-local}
so that the regularized currents are smooth and closed in $V_{\varepsilon}=\left\{ \zeta\in V\mbox{ such that }\delta_{\partial V}(\zeta)>\varepsilon\right\} $.

Thus, to finish the proof of \propref{Local-solution-d-equation}
we assume the currents $\vartheta$ and $\omega$ supported and smooth
in $\overline{V_{\varepsilon}}$ and we will solve the equation $dw=\vartheta$
and $df=\omega$ in $W_{\varepsilon}=B_{e}\left(p,r_{2}-\varepsilon\right)\cap\left\{ \zeta\in\Omega\mbox{ such that }\delta_{\Omega}(\zeta)<\eta_{1}\right\} $
(so that $\cup_{\varepsilon}W_{\varepsilon}=W$), using \propref{comparison-carleson-norms-smooth-forms},
and the \propref{Local-solution-d-equation} will follow a standard
weak limiting procedure.
\begin{proof}[Proof of \propref{Local-solution-d-equation} for smooth currents
in $V_{\varepsilon}$]
 By translation we assume $A(p)=0$.

Let $\left(P_{j}\right)_{j}$ be a minimal covering of $V_{\varepsilon}$
by polydisks centered on $Z_{j}$, $P_{j}=P\left(Z_{j},\delta_{\Omega}\left(Z_{j}\right)\right)$.
Let $\left(\Phi_{j}\right)_{j}$ be a smoth partition of $1$ associated
to the $P_{j}$ i.e. $\Phi_{j}\geq0$, $\sum\Phi_{j}=1$, $\Phi_{j}$
identically zero outside $2P_{j}$, chosen so that $\left|\frac{\partial\Phi_{j}}{\partial v}\right|\lesssim\tau\left(\cdot,v,\delta_{\Omega}\left(Z_{j}\right)\right)^{-1}$.

Let $\left(\psi_{k}\right)_{k\geq0}$ be a family of functions in
$\mathcal{C}^{\infty}(\mathbb{R})$ with support in $\left]2^{-k-1},2^{-k+1}\right]$
and such that $\psi_{k}\geq0$, $\sum\psi_{k}=1$ on $\left]0,1\right]$
and $\psi'_{k}(t)\lesssim2^{k}$.

Finally let $\varphi\in\mathcal{C}^{\infty}(\mathbb{R})$, $0\leq\varphi\leq1$
such that $\varphi(x)=1$ if $x<\nicefrac{1}{2}$ and $\varphi(x)=0$
if $x>1$.

Let us denote $\mathbb{D}$ the unit disk of $\mathbb{C}$. For $\Lambda=\left(\lambda_{1},\ldots,\lambda_{n}\right)\in\mathbb{D}^{n}$
and for $c$ sufficiently small, to be precised, we consider the following
function
\begin{equation}
h_{\Lambda}(t,z)=tz+ct\sum_{k,j}\psi_{k}(1-t)\Phi_{j}(tz)\sum_{i=1}^{n}A_{i,j,k}(z)\label{eq:definition-Hlambda}
\end{equation}
with
\begin{multline*}
A_{i,j,k}(z)=\varphi\left(\frac{2^{-k}}{-\rho(z)}\right)\frac{2^{-k}}{\delta_{\Omega}\left(Z_{j}\right)}\lambda_{i}\tau_{i}\left(Z_{j},\delta_{\Omega}\left(Z_{j}\right)\right)e_{i}\left(Z_{j},\delta_{\Omega}\left(Z_{j}\right)\right)+\\
\left(1-\varphi\left(\frac{2^{-k}}{-\rho(z)}\right)\right)\lambda_{i}\tau_{i}\left(Z_{j},2^{-k}\right)e_{i}\left(Z_{j},2^{-k}\right)
\end{multline*}
where $\left(e_{i}\left(Z_{j},\delta_{\Omega}\left(Z_{j}\right)\right)\right)_{i}$
is the $\left(Z_{j},\delta_{\Omega}\left(Z_{j}\right)\right)$-extremal
basis used for the polydisk $P_{j}$ and $\left(e_{i}\left(Z_{j},2^{-k}\right)\right)_{i}$
a $\left(Z_{j},2^{-k}\right)$-extremal basis.

This function satisfies the following properties:
\begin{stlem}
\label{lem:properties-h-lambda}\mynobreakpar
\begin{enumerate}
\item $h_{\lambda}(0,z)=0$, $h_{\lambda}(1,z)=z$ and $d_{z}h_{\lambda}(1,z)=dz$;
\item If $\Phi_{j}(tz)\neq0$, then $\delta_{\Omega}(tz)\simeq\delta_{\Omega}\left(Z_{j}\right)$
and $tz\in P_{j}\left(Z_{j},K\delta_{\Omega}\left(Z_{j}\right)\right)$;
\item If $\Phi_{j}(tz)\neq0$, if $c_{0}$ ((\ref{eq:def-pseudo-balls}))
and $c$ ((\ref{eq:definition-Hlambda})) are small enough,

\begin{enumerate}
\item if $\varphi\left(\frac{2^{-k}}{-\rho(z)}\right)\neq0$, then $\delta_{\Omega}(h_{\Lambda}(t,z))\simeq\delta_{\Omega}(tz)\simeq\delta_{\Omega}(z)\simeq\delta_{\Omega}\left(Z_{j}\right)$
and $h_{\Lambda}(t,z)\in P\left(Z_{j},K\delta_{\Omega}\left(Z_{j}\right)\right)$,
\item If $\varphi\left(\frac{2^{-k}}{-\rho(z)}\right)=0$, then $\delta_{\Omega}(h_{\Lambda}(t,z))\simeq2^{-k}$
and $h_{\Lambda}(t,z)\in P\left(Z_{j},K2^{-k}\right)$;
\end{enumerate}
\item Denoting
\[
Q(t,z)=\left\{ \begin{array}{cc}
\mathcal{B}(tz,1-t) & \mbox{ if }1-t\geq-\rho(z)\\
\frac{1-t}{-\rho(z)}\mathcal{B}(tz,-\rho(z)) & \mbox{ if }1-t\leq-\rho(z)
\end{array}\right.
\]
and $Q_{1}(t,z)=\left\{ w=h_{\Lambda}(t,z)\mbox{ for }\Lambda\in\mathbb{D}^{n}\right\} $.

\begin{enumerate}
\item Then $c$ in (\ref{eq:definition-Hlambda}) being small enough, for
$z\in W$ and $\forall t\in\left[0,1\right]$, $Q_{1}(t,z)\subset Q(t,z)$,
, $Q_{1}(t,z)\subset B_{e}\left(p,r_{1}\right)\cap\Omega$,
\item $\forall t_{0}>0$, there exists $c_{1}=c_{1}\left(\Omega,c,t_{0}\right)>0$
so that $c_{1}Q(t,z)\subset Q_{1}(t,z)$, $\forall t\geq t_{0}$.
\end{enumerate}
\end{enumerate}
\end{stlem}

The homotopy operator $H$ for $d$ is then defined on smooth forms
$\vartheta$ taking the average over $\mathbb{D}^{n}$ of the integrals
of the $dt$-component of $h_{\Lambda}^{*}\vartheta$: for example,
if $\vartheta$ is a smooth form of degree $2$, $H(\vartheta)=H_{\vartheta}$
is the form of degree $1$ 
\[
H_{\vartheta}(z)=\oint H_{\Lambda}(\vartheta)(z)d\Lambda,
\]
where $H_{\Lambda}(\vartheta)(z)$ is defined on every vector $v$
by
\[
H_{\Lambda}(\vartheta)(z)(v)=\int_{0}^{1}h_{\Lambda}^{*}\vartheta dt=\int_{0}^{1}\vartheta\left(h_{\Lambda}(t,z)\right)\left(Y_{t},Z_{t,v}\right)dt,
\]
where $Y_{t}=\frac{\partial}{\partial t}h_{\Lambda}(t,z)$ and $Z_{t,v}=\frac{\partial}{\partial v}h_{\Lambda}(t,z)$.
Then $H(\vartheta)$ is smooth and $dH+Hd=Id$.

To get the required Carleson estimate for $H(\vartheta)$, $\vartheta$
and $\omega$ must be zero around the origin $A(p)$, so we follow
a classical procedure (see \cite{AnCa-Hp-90}). Let $R>0$ such that
$B(0,2R)\subset V_{\varepsilon}\setminus W_{\varepsilon}$ ($\varepsilon>0$
small). Let $\psi$ be a smooth cut of function equal to $0$ in $B(0,R)$
and $1$ in $V_{\varepsilon}\setminus B(0,2R)$. Then, with $T=\vartheta\mbox{ or }\omega$,
we have 
\[
dH(\psi T)=\psi T-H(d\psi\wedge T),
\]
and $H(d\psi\wedge T)$ is closed in $W_{\varepsilon}$ and, if $d\tau=H(d\psi\wedge T)$
then $d(H(\psi T)+\tau)=\vartheta$ in $W_{\varepsilon}$.

Then the conclusion follows the next proposition because, using a
standard Poincaré homotopy (reducing eventually $W_{\varepsilon}$),
$\tau$ can be chosen satisfying 
\[
\left\Vert \chi_{W_{\varepsilon}}\delta_{\Omega}^{\nicefrac{1}{m}-1}|\tau|\right\Vert _{W^{1}(\Omega)}\lesssim\left\Vert \tau\right\Vert _{L^{\infty}\left(W_{\varepsilon}\right)}\leq\left\Vert H(d\psi\wedge T)\right\Vert _{L^{\infty}\left(W_{\varepsilon}\right)}.
\]

\begin{spprop}
There exists a constant $C$ such that:
\begin{enumerate}
\item with the above notation, $\left\Vert H(d\psi\wedge T)\right\Vert _{L^{\infty}\left(W_{\varepsilon}\right)}\leq C\left\Vert T\right\Vert _{W^{1}(\Omega)}$;
\item if $\vartheta$ is a smooth current of degree $2$ identically zero
in $B\left(0,R\right)$, then $\left\Vert \chi_{W_{\varepsilon}}\left\Vert H(\vartheta)\right\Vert _{\mathbbmss k}d\lambda\right\Vert _{W^{1}(\Omega)}\leq C\left\Vert \delta_{\Omega}\left\Vert \vartheta\right\Vert _{\mathbbmss k}d\lambda\right\Vert _{W^{1}(\Omega)}$;
\item if $\omega$ is a smooth current of degree $1$ identically zero in
$B\left(0,R\right)$, then $\left\Vert \chi_{W_{\varepsilon}}\delta_{\Omega}^{\nicefrac{1}{m}-1}H(\omega)d\lambda\right\Vert _{W^{1}(\Omega)}\leq C\left\Vert \left\Vert \omega\right\Vert _{\mathbbmss k}d\lambda\right\Vert _{W^{1}(\Omega)}$;
\end{enumerate}
\end{spprop}

\begin{proof}
We begin to prove (2). Note that there exists $t_{0}>0$ such that,
$\vartheta\left(h_{\Lambda}(t,z)\right)\neq0$ implies $t>t_{0}$
so that
\[
H_{\Lambda}(\vartheta)(z)(v)=\int_{t_{0}}^{1}\vartheta\left(h_{\Lambda}(t,z)\right)\left(Y_{t},Z_{t,v}\right)dt.
\]

\begin{llemsp}
\label{lem:lemma1-proof-prop41}Let $T$ be the operator defined on
non negative functions $f$ on $V_{\varepsilon}$, by 
\[
T(f)(z)=\delta_{\Omega}(z)^{\nicefrac{1}{m}-1}\int_{t_{0}}^{1}(1-t)^{\nicefrac{1}{m}-1}\left(\oint_{Q(t,z)}f(w)\delta_{\Omega}^{1-\nicefrac{2}{m}}(w)d\lambda(w)\right)dt,\,z\in W_{\varepsilon}.
\]
Then, for $z\in W_{\varepsilon}$,
\[
\left\Vert H_{\vartheta}(z)\right\Vert _{\mathbbmss k}\lesssim T\left(\delta_{\Omega}(\cdot)\left\Vert \vartheta(\cdot)\right\Vert _{\mathbbmss k}\right)(z).
\]
\end{llemsp}

\begin{proof}[Proof of the lemma]
 By the definition of the ``norm'' $\left\Vert \cdot\right\Vert _{\mathbbmss k}$
we have
\[
\left|H_{\vartheta}(z)(v)\right|\leq\oint\left[\int_{t_{0}}^{1}\left\Vert \vartheta\left(h_{\Lambda}(t,z)\right)\right\Vert _{\mathbbmss k}k\left(h_{\Lambda}(t,z),Y_{t}\right)k\left(h_{\Lambda}(t,z),Z_{t,v}\right)dt\right]d\Lambda.
\]
By the definition of $h_{\Lambda}(t,z)$
\begin{multline*}
Y_{t}=z+c\sum_{k,j}\psi_{k}(1-t)\Phi_{j}(tz)\sum_{i=1}^{n}A_{i,j,k}(z)+\\
ct\left[\sum_{k,j}\left\langle d\Phi_{j}(tz);z\right\rangle \psi_{k}(1-t)-\Phi_{j}(tz)\psi_{k}'(1-t)\right]\sum_{i}A_{i,j,k}(z).
\end{multline*}

Assume $\Phi_{j}(tz)\neq0$ and $\psi_{k}(1-t)\neq0$.

Let us first estimate $k\left(h_{\Lambda}(t,z),Y_{t}\right)$. We
have $\left\langle d\Phi_{j}(tz);z\right\rangle \lesssim\delta_{\Omega}\left(Z_{j}\right)^{-1}\simeq\delta_{\Omega}(tz)^{-1}\lesssim\frac{1}{1-t}$,
$\left|\psi_{k}'(1-t)\right|\lesssim\frac{1}{1-t}$ and, by (3) (a)
of \lemref{properties-h-lambda}, 
\[
\tau\left(h_{\Lambda}(t,z),e_{i}\left(Z_{j},2^{-k}\right),\delta_{\Omega}\left(h_{\Lambda}(t,z)\right)\right)\gtrsim\tau\left(Z_{j},e_{i}\left(Z_{j},2^{-k}\right),\delta_{\Omega}\left(h_{\Lambda}(t,z\right)\right).
\]
As $s>1$ implies $\tau(p,v,s\delta)\gtrsim s^{\nicefrac{1}{m}}\tau(p,v,\delta)$
we get, if $2^{-k}\leq-\rho(z)$,
\[
\tau_{i}\left(Z_{j},2^{-k}\right)k\left(h_{\Lambda}(t,z),e_{i}\left(Z_{j},2^{-k}\right)\right)\lesssim\delta_{\Omega}\left(h_{\Lambda}(t,z)\right)\left(\frac{\delta_{\Omega}\left(h_{\Lambda}(t,z)\right)}{2^{-k}}\right)^{-\nicefrac{1}{m}}.
\]
Similarly, if $2^{-k}\geq-\nicefrac{\rho(z)}{2}$,
\[
\tau_{i}\left(Z_{j},\delta_{\Omega}\left(Z_{j}\right)\right)k\left(h_{\Lambda}(t,z),e_{i}\left(Z_{j},\delta_{\Omega}\left(Z_{j}\right)\right)\right)\lesssim\delta_{\Omega}\left(Z_{j}\right)\simeq\delta_{\Omega}\left(h_{\Lambda}(t,z)\right),
\]
and (because $\delta_{\Omega}\left(h_{\Lambda}(t,z)\right)\simeq2^{-k}$)
\[
\frac{2^{-k}}{\delta_{\Omega}\left(Z_{j}\right)}\tau_{i}\left(Z_{j},\delta_{\Omega}\left(Z_{j}\right)\right)k\left(h_{\Lambda}(t,z),e_{i}\left(Z_{j},\delta_{\Omega}\left(Z_{j}\right)\right)\right)\lesssim\delta_{\Omega}\left(h_{\Lambda}(t,z)\right)\left(\frac{\delta_{\Omega}\left(h_{\Lambda}(t,z)\right)}{2^{-k}}\right)^{-\nicefrac{1}{m}}.
\]
These estimates give
\[
k\left(h_{\Lambda}(t,z),Y_{t}\right)\lesssim\left(\frac{\delta_{\Omega}\left(h_{\Lambda}(t,z)\right)}{2^{-k}}\right)^{1-\nicefrac{1}{m}}\simeq\left(\frac{\delta_{\Omega}\left(h_{\Lambda}(t,z)\right)}{1-t}\right)^{1-\nicefrac{1}{m}}.
\]
The estimate $k\left(h_{\Lambda}(t,z),Z_{t,v}\right)$ is easy: recording
that ((1) of \lemref{basic-properties-geometry}) 
\[
\left|\frac{\partial\Phi_{j}(\cdot)}{\partial v}\right|\lesssim\tau\left(\cdot,v,\delta_{\Omega}\left(Z_{j}\right)\right)^{-1},\,\left|\frac{\partial}{\partial v}(-\rho)(z)\right|\lesssim\frac{\delta_{\Omega}(z)}{\tau\left(z,v,\delta_{\Omega}(z)\right)}
\]
and that $\varphi'\left(\frac{2^{-k}}{-\rho(z)}\right)\neq0$ implies
$2^{-k}\simeq-\rho(z)$, one easily gets
\[
k\left(h_{\Lambda}(t,z),Z_{t,v}\right)\lesssim\frac{\delta_{\Omega}(z)^{\nicefrac{1}{m}}\delta_{\Omega}\left(h_{\Lambda}(t,z)\right)^{1-\nicefrac{1}{m}}}{\tau\left(z,v,\delta_{\Omega}(z)\right)}.
\]

Then we obtain
\[
\left\Vert H(\vartheta)(z)\right\Vert _{\mathbbmss k}\lesssim\delta_{\Omega}(z)^{\nicefrac{1}{m}-1}\oint d\Lambda\int_{0}^{1}\left\Vert \vartheta\left(h_{\Lambda}(t,z)\right)\right\Vert _{\mathbbmss k}\delta_{\Omega}\left(h_{\Omega}(t,z)\right)^{2-\nicefrac{2}{m}}\left(1-t\right)^{\nicefrac{1}{m}-1},
\]
and the lemma is obtained making the change of variables $\Lambda\mapsto h_{\Lambda}(t,z)$,
the jacobian being proportional to the volume of $Q_{1}(t,z)$ which
is equivalent to $Q(t,z)$ because $t\geq t_{0}$ implies $c{}_{1}Q(t,z)\subset Q_{1}(t,z)\subset Q(t,z)$.\end{proof}
\begin{llemsp}
The operator $T$ of \lemref{lemma1-proof-prop41} satisfies the following
estimate
\[
\left\Vert \chi_{W_{\varepsilon}}T(f)d\lambda\right\Vert _{W^{1}(\Omega)}\lesssim\left\Vert \chi_{V_{\varepsilon}}fd\lambda\right\Vert _{W^{1}(\Omega)}.
\]
\end{llemsp}

\begin{proof}[Proof of the lemma]
Let $B(\xi,\varepsilon)$ be a pseudo-ball on $\partial\Omega$ and
$\widehat{B}(\xi,\varepsilon)$ the tent over $B(\xi,\varepsilon)$.
For $z\in\widehat{B}(\xi,\varepsilon)$ we decompose $T(f)$ into
two pieces (to simplify the notation we write $\left\Vert fd\lambda\right\Vert _{W^{1}}$
instead of $\left\Vert \chi_{V_{\varepsilon}}fd\lambda\right\Vert _{W^{1}(\Omega)}$):
\[
T_{1}(f)(z)=\delta_{\Omega}(z)^{\nicefrac{1}{m}-1}\int_{t_{0}}^{1-\varepsilon}(1-t)^{\nicefrac{1}{m}-1}\left(\oint_{Q(t,z)}f(w)\delta_{\Omega}^{1-\nicefrac{2}{m}}(w)d\lambda(w)\right)dt,
\]
and 
\[
T_{2}(f)(z)=\delta_{\Omega}(z)^{\nicefrac{1}{m}-1}\int_{1-\varepsilon}^{1}(1-t)^{\nicefrac{1}{m}-1}\left(\oint_{Q(t,z)}f(w)\delta_{\Omega}^{1-\nicefrac{2}{m}}(w)d\lambda(w)\right)dt.
\]

Consider first $T_{1}(f)$. As $t<1-\varepsilon$ and $\delta_{\Omega}(z)\lesssim\varepsilon$,
we have $\delta_{\Omega}\left(w\right)\simeq\delta_{\Omega}(tz)\simeq1-t+\delta_{\Omega}(z)\simeq1-t$,
$d\left(w,\xi\right)\lesssim1-t$ and $Q(t,z)\subset\widehat{B}(\xi,K(1-t))$.
Then (note that $\widehat{B}(\xi,1-t)$$\subset KQ(t,z)$, because
$-\rho(z)\simeq\delta_{\Omega}(z)\lesssim\varepsilon\lesssim1-t$)
\[
\int_{Q(t,z)}f(w)d\lambda(w)\lesssim\left\Vert fd\lambda\right\Vert _{W^{1}}\frac{\mbox{Vol}\left(\widehat{B}(\xi,1-t)\right)}{1-t}\lesssim\left\Vert fd\lambda\right\Vert _{W^{1}}\frac{\mbox{Vol}\left(Q(t,z)\right)}{1-t},
\]
and, using $\delta_{\Omega}\left(w\right)\simeq1-t$, we get
\[
T_{1}(f)(z)\lesssim\left\Vert fd\lambda\right\Vert _{W^{1}}\delta_{\Omega}(z)^{\nicefrac{1}{m}-1}\int_{t_{0}}^{1-\varepsilon}(1-t)^{-\nicefrac{1}{m}-1}dt\lesssim\left\Vert fd\lambda\right\Vert _{W^{1}}\delta_{\Omega}(z)^{\nicefrac{1}{m}-1}\varepsilon^{-\nicefrac{1}{m}}
\]
and
\[
\int_{\widehat{B}(\xi,\varepsilon)}T_{1}(f)(z)d\lambda(z)\lesssim\left\Vert fd\lambda\right\Vert _{W^{1}}\varepsilon^{-\nicefrac{1}{m}}\int_{\widehat{B}(\xi,\varepsilon)}\delta_{\Omega}(z)^{\nicefrac{1}{m}-1}d\lambda(z)\lesssim\left\Vert fd\lambda\right\Vert _{W^{1}}\sigma(B(\xi,\varepsilon)).
\]

Consider now $T_{2}(f)$. $\widehat{B}(\xi,\varepsilon)$ is equivalent
to the set 
\[
\left\{ r\eta\mbox{ such that }1-\varepsilon\leq r\leq1\mbox{ and }\eta\in B(\xi,\varepsilon)\right\} ,
\]
and if $z=r\eta$, $\delta_{\Omega}(z)\simeq1-r$, and for $w\in Q(t,z)$,
$w\in\widehat{B}(\xi,K\varepsilon)$ and $\delta_{\Omega}(w)\simeq\delta_{\Omega}(tz)\simeq1-tr$.
Then
\begin{multline*}
I=\int_{\widehat{B}(\xi,\varepsilon)}T_{2}(f)(z)d\lambda(z)\lesssim\\
\int_{\widehat{B}(\xi,K\varepsilon)}f(w)\delta_{\Omega}(w)^{-\nicefrac{1}{m}}\left(\int_{\mathcal{D}_{w}}\frac{(1-r)^{\nicefrac{1}{m}-1}(1-t)^{\nicefrac{1}{m}-1}(1-tr)^{1-\nicefrac{1}{m}}}{\mbox{Vol}\left(Q(t,r\eta)\right)}drdtd\sigma(\eta)\right)d\lambda(w)
\end{multline*}
where
\[
\mathcal{D}_{w}=\left\{ (t,r,\eta)\mbox{ such that }w\in Q(t,r\eta),\,(t,r)\in\left[1-\varepsilon,1\right]^{2}\mbox{ and }\eta\in B(\xi,\varepsilon)\right\} .
\]

Note that $1-tr\simeq\max\left\{ (1-t),(1-r)\right\} $ and let us
cut $\mathcal{D}_{w}$ into two parts
\[
\mathcal{D}_{w}^{1}=\mathcal{D}_{w}\cap\left\{ r\geq t\right\} \mbox{ and }\mathcal{D}_{w}^{2}=\mathcal{D}_{w}\cap\left\{ r<t\right\} ,
\]
and define $I_{i}$, $i=1,2$, replacing in the definition of $I$
$\mathcal{D}_{w}$ by $\mathcal{D}_{w}^{i}$.

If $\left(t,r,\eta\right)\in\mathcal{D}_{w}^{1}$, $\delta_{\Omega}(tr\eta)\simeq1-t$
and $\delta_{\Omega}(w)\simeq1-t$. Let $w_{1}$ be the intersection
of $\partial\Omega$ with the half real line passing through $0$
and $w$. Then $\eta\in B\left(w_{1},K_{1}(1-t)\right)$, $t\in\left[1-K_{2}\delta_{\Omega}(w),1-c_{2}\delta_{\Omega}(w)\right]$
and $r\geq t\geq1-K_{2}\delta_{\Omega}(w)$. As (by (2) of \propref{comparison-carleson-norms-smooth-forms}
and (\ref{eq:equivalence-pseudoballs-polydisk})) 
\[
\mbox{Vol}\left(Q(t,r\eta)\right)\simeq\mbox{Vol}\left(\mathcal{B}(w,1-t)\right)\simeq(1-t)\sigma\left(B\left(w_{1},1-t\right)\right)\simeq(1-t)\sigma\left(B\left(w_{1},K_{1}(1-t)\right)\right),
\]
we get
\begin{eqnarray*}
I_{1} & \lesssim & \int_{\widehat{B}(\xi,K\varepsilon)}f(w)\delta_{\Omega}(w)^{-\nicefrac{1}{m}}\left(\int_{\mathcal{D}_{w}^{1}}\frac{(1-r)^{\nicefrac{1}{m}-1}}{(1-t)\sigma\left(B\left(w_{1},K_{1}(1-t)\right)\right)}drdt\sigma(\eta)\right)d\lambda(w)\\
 & \lesssim & \int_{\widehat{B}(\xi,K\varepsilon)}f(w)\delta_{\Omega}(w)^{-\nicefrac{1}{m}}\int_{1-K_{2}\delta_{\Omega}(w)}^{1-c_{2}\delta_{\Omega}(w)}\left(\int_{B\left(w_{1},K_{1}(1-t)\right)}\frac{d\sigma(\eta)}{\sigma\left(B\left(w_{1},K_{1}(1-t)\right)\right)}\right.\\
 &  & \hspace{4.7cm}\left.\int_{1-K_{2}\delta_{\Omega}(w)}^{1}(1-r)^{\nicefrac{1}{m}-1}dr\right)\frac{dt}{1-t}d\lambda(w)\\
 & \lesssim & \int_{\widehat{B}(\xi,K\varepsilon)}f(w)d\lambda(w)\lesssim\left\Vert fd\lambda\right\Vert _{W^{1}}\sigma\left(B(\xi,K\varepsilon)\right)\lesssim\left\Vert fd\lambda\right\Vert _{W^{1}}\sigma\left(B(\xi,\varepsilon)\right).
\end{eqnarray*}

Finally, if $\left(t,r,\eta\right)\in\mathcal{D}_{w}^{2}$, $\delta_{\Omega}(w)\simeq1-r$
and
\begin{itemize}
\item if $\frac{1-t}{-\rho(r\eta)}\leq1$, then $w\in\frac{1-t}{-\rho(r\eta)}\mathcal{B}(tr\eta,-\rho(r\eta))\subset\frac{1-t}{1-r}\mathcal{B}(tr\eta,1-r)$
so $tr\eta\in K_{1}\frac{1-t}{1-r}\mathcal{B}(w,1-r)$, and, moreover
\[
\mbox{Vol}\left(\frac{1-t}{-\rho(r\eta)}\mathcal{B}(tr\eta,-\rho(r\eta))\right)\simeq\mbox{Vol}\left(\frac{1-t}{1-r}\mathcal{B}(w,1-r)\right);
\]

\item if $\frac{1-t}{-\rho(r\eta)}\geq1$, then $w\in\mathcal{B}(tr\eta,1-t)\subset K'\frac{1-t}{1-r}\mathcal{B}(w,1-r)$,
because in this case $1-t\simeq1-r$, so $tr\eta\in K_{1}\frac{1-t}{1-r}\mathcal{B}(w,1-r)$,
and
\[
\mbox{Vol}(\mathcal{B}(tr\eta,1-r))\simeq\mbox{Vol (}\mathcal{B}(w,1-r)).
\]

\end{itemize}
Now, for $t$, $r$ and $w$ fixed, $\sigma\left(\left\{ \eta\mbox{ such that }(t,r,\eta)\in\mathcal{D}_{w}^{2}\right\} \right)\lesssim\frac{1}{1-t}\mbox{Vol}\left(\frac{1-t}{1-r}\mathcal{B}(w,1-r)\right)$,
and, $\frac{1-t}{1-r}\mathcal{B}(w,1-r)\subset K_{2}\mathcal{B}(w,1-t)$
by (\ref{eq:equivalence-pseudoballs-polydisk}) and (\ref{eq:comp-tau-epsilon-tau-lambda-epsilon}),
and, by the last property of $W$, for $w$ and $t$ fixed, the lengh
of the set of $r$ such that there exists $\eta$ such that $tr\eta\in\mathcal{D}_{w}^{2}$
is $\lesssim1-t$.

Then 
\begin{multline*}
I_{2}\lesssim\int_{\widehat{B}(\xi,K\varepsilon)}f(w)\delta_{\Omega}(w)^{-\nicefrac{1}{m}}\\
\int_{1-K_{3}\delta_{\Omega}(w)}^{1}\left(\int_{\left\{ (r,\eta)\mbox{ s. t. }tr\eta\in\mathcal{D}_{w}^{2}\right\} }\frac{d\sigma(\eta)dr}{\mbox{Vol}\left(\frac{1-t}{1-r}\mathcal{B}\left(w,1-t\right)\right)}\right)(1-t)^{\nicefrac{1}{m}-1}dtd\lambda(w),
\end{multline*}
and we get $I_{2}\lesssim\int_{\widehat{B}(\xi,K\varepsilon)}f(w)d\lambda(w)\lesssim\left\Vert fd\lambda\right\Vert _{W^{1}}\sigma\left(B(\xi,\varepsilon)\right)$
finishing the proof of the lemma.
\end{proof}

(2) of the proposition follows immediately the lemmas.

Assertion (3) of the proposition is proved in a similar and easier
way. We will not give more details.

We finish giving briefly the proof of (1) of the proposition.There
exists $\delta>0$ such that $t<\delta$ or $t>1-\delta$ implies
$d\psi\wedge\vartheta\left(h_{\Lambda}(z,t)\right)=0$ so
\[
H(d\psi\wedge\vartheta)(z)\left(v_{1},v_{2}\right)=\oint_{\Lambda}\int_{\delta}^{1-\delta}d\psi\wedge\vartheta\left(h_{\Lambda}(z,t)\right)\left(Y_{t},Z_{t,v_{1}},Z_{t,v_{2}}\right)dt.
\]

Note that, for $t\in\left[\delta,1-\delta\right]$ and $\left|v_{i}\right|\leq1$,
$\left|Z_{t,v_{i}}\right|$ and $\left|Y_{t}\right|$ are bounded
by $C=C(\delta)$, $i=1,2$. Then after the change of variables $w(\Lambda)=h_{\Lambda}(z,t)$
we get
\begin{eqnarray*}
\left|H(d\psi\wedge\vartheta)(z)\left(v_{1},v_{2}\right)\right| & \lesssim & \int_{\delta}^{1-\delta}dt\int_{Q(t,z)}\frac{\left|d\psi\wedge\vartheta\right|(w)}{\mbox{Vol}\left(Q(t,z)\right)}d\lambda(w)\\
 & \lesssim & \int_{\delta}^{1-\delta}\int_{Q(t,z)}\delta_{\Omega}(w)\left|d\psi\wedge\vartheta\right|(w)d\lambda(w)dt\\
 & \lesssim & \int_{\widehat{B}}\delta_{\Omega}(w)\left|d\psi\wedge\vartheta\right|(w)d\lambda(w)\\
 & \lesssim & \left\Vert \delta_{\Omega}d\psi\wedge\vartheta\right\Vert _{W^{1}}\lesssim\left\Vert \delta_{\Omega}\vartheta\right\Vert _{W^{1}}
\end{eqnarray*}
the first inequality coming from the fact that $\mbox{Vol}\left(Q(t,z)\right)$
and $\delta_{\Omega}(w)$ are bounded from below, $\widehat{B}$ in
the third inequality being a tent containing $Q(t,z)$ and the last
because $\sigma(B)$ is bounded from below.
\end{proof}
The proof of \propref{Local-solution-d-equation} is now complete.
\end{proof}

\section{Proof of Theorem \ref{thm:d-bar_resolution}\label{sec:Proof-of-d-bar-thm}}

Let us introduce the notion of $BMO_{s}$ functions, similarly to
$s$-Carleson measures and currents defined before \propref{regularization-global}:

$f$ is in $BMO_{s}(\partial D)$ if 
\[
\left\Vert f\right\Vert _{BMO_{s}(\partial D)}:=\sup_{z\in\partial D,\,s<t<\varepsilon_{0}}\int_{P_{t}(z)\cap\partial D}\left|f-\oint_{P_{t}(z)\cap\partial D}f\right|<+\infty.
\]

It is easy to see that, if $\omega$ is \emph{smooth}, then, the proof
of \propref{comparison-carleson-norms-smooth-forms} shows that
\[
\left\Vert \omega\right\Vert _{W_{s}^{1}(D)}\simeq\left\Vert \left\Vert \omega\right\Vert _{\mathbbmss k}d\lambda\right\Vert _{W_{s}^{1}(D)}.
\]

We start the proof of the theorem regularizing the current $\omega$
using \propref{regularization-global}: then (simplifying the notations),
for $\varepsilon>0$ small enough, the regularized current $\omega_{\varepsilon}$
is smooth $\overline{\partial}$-closed in $\Omega^{\varepsilon}=\left\{ \rho<-C\varepsilon\right\} $
and satisfies $\left\Vert \omega_{\varepsilon}\right\Vert _{W_{s}^{1}\left(\Omega^{\varepsilon}\right)}\lesssim\left\Vert \omega\right\Vert _{W^{1}(\Omega)}$.

Now we solve the equation $\overline{\partial}u_{\varepsilon}=\omega_{\varepsilon}$
using the method described in the proof of \cite[Theorem 2.4]{ChDu18}:
$u_{\varepsilon}$ is given by the formula
\[
u_{\varepsilon}(z)=\int_{\Omega^{\varepsilon}}K_{\varepsilon}^{1}(z,\zeta)\wedge\omega_{\varepsilon}(\zeta)-\overline{\partial}*\mathcal{N}_{\varepsilon}\left(\int_{\Omega^{\varepsilon}}P_{\varepsilon}(z,\zeta)\wedge\omega_{\varepsilon}(\zeta)\right)
\]
the kernels $K_{\varepsilon}^{1}$ and $P_{\varepsilon}$ are associated
to the defining function $\rho+C\varepsilon$ as described in \cite{CDMb,ChDu18}.
We proceed as H Skoda in \cite[Section 8]{Sko76}. Clearly, the $\mathcal{C}^{\infty}$-smooth
kernels $P_{\varepsilon}$, as well as their derivatives, converge
uniformly to the corresponding kernel $P$ of $\Omega$ so that $\int_{\Omega^{\varepsilon}}P_{\varepsilon}(z,\zeta)\wedge\omega_{\varepsilon}(\zeta)$
converges in every Sobolev norm to a $\overline{\partial}$-closed
form $g$ on $\Omega$ such that, for all integer $k$, $\left\Vert g\right\Vert _{H^{k}}$
is bounded by the total mass of $\omega$. Then 
\[
\overline{\partial}^{*}\mathcal{N}_{\varepsilon}\left(\int_{\Omega^{\varepsilon}}P_{\varepsilon}(z,\zeta)\wedge\omega_{\varepsilon}(\zeta)\right)
\]
 converges in $\mathcal{C}^{1}\left(\overline{\Omega}\right)$ to
a function $h$.

Let $\Phi_{\varepsilon}:\partial\Omega\rightarrow\partial\Omega^{\varepsilon}$
be a family of $\mathcal{C}^{\infty}$ diffeomorphisms such that $\Phi_{\varepsilon}$
converges to the identity uniformly in $\mathcal{C}^{\infty}$ norm
on $\partial\Omega$.

As $\overline{\partial}u_{\varepsilon}=\omega_{\varepsilon}$, denoting
$v_{\varepsilon}=\int_{\Omega^{\varepsilon}}K_{\varepsilon}^{1}(z,\zeta)\wedge\omega_{\varepsilon}(\zeta)$,
if $v_{\varepsilon}\circ\Phi_{\varepsilon}$ converges in $L^{1}(\partial\Omega)$
to $v=\int_{\Omega}K^{1}(z,\zeta)\wedge\omega(\zeta)$, for $z\in\partial\Omega$,
the function
\[
u(z)=\int_{\Omega}K^{1}(z,\zeta)\wedge\omega(\zeta)-h
\]
is a solution of the equation $\overline{\partial}_{b}u=\omega$.
By the properties of the kernels $K_{\varepsilon}^{1}$ and $K^{1}$
($K_{\varepsilon}^{1}$ converges uniformly on $\partial\Omega\times\overline{\Omega_{\eta}}$
($\eta>0$ fixed) to $K^{1}$) this convergence follows exactly the
proof made by H Skoda in \cite[p. 272]{Sko76}.

To conclude the proof of \thmref{d-bar_resolution} we have to show
that $\left\Vert v\right\Vert _{BMO(\partial\Omega)}\lesssim\left\Vert \omega\right\Vert _{W^{1}(\Omega)}$.
The proof of \cite[Theorem 2.4]{ChDu18} gives
\[
\left\Vert u_{\varepsilon}\right\Vert _{BMO_{\varepsilon}}\leq C_{1}\left\Vert \left\Vert \omega\right\Vert _{\mathbbmss k}\right\Vert _{W_{\varepsilon}^{1}\left(\Omega^{\varepsilon}\right)}\lesssim\left\Vert \omega\right\Vert _{W^{1}(\Omega)}
\]
with a constant $C_{1}$ uniform in $\varepsilon$ (small enough)
because the estimates of \cite[lemmas 3.4, 3.5 and 3.6]{ChDu18} are
uniform in a neighborhood of $\partial\Omega$. Thus the end of the
proof is the following lemma
\begin{sllem}
With the previous notations $\left\Vert u\right\Vert _{BMO(\partial\Omega)}\lesssim\sup_{\varepsilon}\left\Vert u_{\varepsilon}\right\Vert _{BMO_{\varepsilon}\left(\partial\Omega^{\varepsilon}\right)}$.\end{sllem}

\begin{proof}
Let $\xi\in\partial\Omega$ and let $B(\xi,t)$ be a pseudo-ball on
$\partial\Omega$. Then $\sigma_{\varepsilon}\left(B\left(\Phi_{\varepsilon}(\xi),t\right)\right)$
converges to $\sigma\left(B(\xi,t)\right)$ and
\[
\oint_{B\left(\Phi_{\varepsilon}(\xi),t\right)}u_{\varepsilon}=\frac{1}{\sigma_{\varepsilon}\left(B\left(\Phi_{\varepsilon}(\xi),t\right)\right)}\int_{B(\xi,t)}u_{\varepsilon}\circ\Phi_{\varepsilon}\left|\mbox{J}\Phi_{\varepsilon}\right|\stackrel{\varepsilon\rightarrow0}{\longrightarrow}\oint_{B(\xi,t)}u.
\]

The lemma follows easily.
\end{proof}

\bibliographystyle{amsalpha}

\providecommand{\bysame}{\leavevmode\hbox to3em{\hrulefill}\thinspace}
\providecommand{\MR}{\relax\ifhmode\unskip\space\fi MR }
\providecommand{\MRhref}[2]{%
  \href{http://www.ams.org/mathscinet-getitem?mr=#1}{#2}
}
\providecommand{\href}[2]{#2}

\end{document}